\documentclass[12pt, reqno]{amsart}
\usepackage{amsmath, amsthm, amscd, amsfonts, amssymb, graphicx, color}
\usepackage[bookmarksnumbered, colorlinks, plainpages]{hyperref}
\hypersetup{colorlinks=true,linkcolor=red, anchorcolor=green, citecolor=cyan, urlcolor=red, filecolor=magenta, pdftoolbar=true}
\input{mathrsfs.sty}

\newtheorem{theorem}{Theorem}[section]
\newtheorem{lemma}[theorem]{Lemma}

\newtheorem{cor}[theorem]{Corollary}
\theoremstyle{definition}

\theoremstyle{remark}
\newtheorem{remark}[theorem]{\bf{Remark}}
\numberwithin{equation}{section}
\begin{document}

\title [Numerical radius inequalities and estimation of zeros of polynomials]  {{ Numerical radius inequalities and estimation of zeros of polynomials }}
\author[S. Jana P. Bhunia and K. Paul] {Suvendu Jana$^1$, Pintu Bhunia$^2$ \MakeLowercase{and} Kallol Paul$^2$}

\address{$^1$ Department of Mathematics, Mahisadal Girls' College, Purba Medinipur 721628, West Bengal, India}
\email{janasuva8@gmail.com}

\address{$^2$ Department of Mathematics, Jadavpur University, Kolkata 700032, West Bengal, India}
\email{pintubhunia5206@gmail.com}
\email{kalloldada@gmail.com; kallol.paul@jadavpuruniversity.in}

\thanks{Pintu Bhunia would like to thank UGC, Govt. of India for the financial support in the form of SRF under the mentorship of Prof. Kallol Paul. }


\maketitle

\begin{abstract}
	Let $A$ be a bounded linear operator defined  on a complex Hilbert space and let $|A|=(A^*A)^{1/2}$ be the positive square root of $A$.
	 Among other refinements of the well known numerical radius inequality $w^2(A)\leq \frac12 \|A^*A+AA^*\|$, 
	 we show that 
	 \begin{eqnarray*}
		w^2(A)&\leq&\frac{1}{4} w^2 \left(|A|+i|A^*|\right)+\frac{1}{8}\left\||A|^2+|A^*|^2\right \|+\frac{1}{4}w\left(|A||A^*|\right) \\
		&\leq& \frac12 \|A^*A+AA^*\|.
	\end{eqnarray*}
Also, we develop inequalities involving numerical radius and spectral radius for the sum of the product operators, from which we derive the following inequalities $$ w^p(A) \leq \frac{1}{\sqrt{2} } w(|A|^p+i|A^*|^p  )\leq \|A\|^p$$ for all $p\geq 1.$ Further, we derive new bounds for the zeros of complex polynomials.

\end{abstract}

\section{Introduction}

\noindent Let $\mathscr{H}$ be a complex Hilbert space with usual inner product $\langle \cdot,\cdot \rangle $ and the corresponding norm $\|\cdot\|$ induced by the inner product.  Let $ \mathbb{B}(\mathscr{H})$ denote the $C^*$-algebra of all bounded linear operators on  $\mathscr{H}$. For $A\in \mathbb{B}(\mathscr{H})$,  $|A|=({A^*A})^{{1}/{2}}$ is the positive square root of $A$. The numerical range of $A$, denoted as $W(A)$, is defined by $W(A)=\left \{\langle Ax,x  \rangle: x\in \mathscr{H}, \|x\|=1 \right \}.$
Let $\|A\|$, $r(A)$ and $w(A)$ denote the operator norm, the spectral radius and the numerical radius of $A$, respectively. Recall that $w(A)=\sup \left \{|\langle Ax,x  \rangle|: x\in \mathscr{H}, \|x\|=1 \right \}.$
The numerical radius $ w(\cdot)$ defines a norm on $\mathbb{B}(\mathscr{H})$, (is equivalent to the operator norm $\|\cdot\|$) is satisfying the following  inequality
\begin{eqnarray}\label{eqv}
\frac{1}{2} \|A\|\leq w({A})\leq\|A\|.
\end{eqnarray}
The first inequality becomes equality if $A^2=0$ and the second one turns into equality if $A$ is normal. 
Similar as the operator norm, numerical radius also satisfies  the power inequality:
\begin{eqnarray}
	w(A^n)\leq w^n(A) \,\, \text {for every $n=1,2,3,\ldots$}.
	\label{eqn7}\end{eqnarray}
It is well known that for $A\in\mathbb{B}(\mathscr{H})$,\begin{eqnarray}
 r(A)\leq w(A).\label{eqn1}\end{eqnarray}
  The inequality (\ref{eqn1}) is sharp. In fact, if $A$ is normal, then $r(A)=w(A)=\|A\|$. For $A,B \in \mathbb{B}(\mathscr{H})$, we have $ r(AB)=r(BA)$ and $ r(A^n)=r^n(A)$ for every positive integer $n$.  
  \smallskip
  Over the years many eminent mathematicians have studied various refinements of (\ref{eqv}) and obtained various bounds for the zeros of a complex polynomial, we refer the readers to  \cite{BhuBook,bib1, BP_RM, aab, a15, a7, SAH, SEO} and the references therein. In \cite{E}, Kittaneh improved the inequalities in (\ref{eqv}) to prove that 
\begin{eqnarray}
\frac{1}{4}\|A^*A+A{A}^*\|\leq w^2({A})\leq\frac{1}{2}\|A^*A+A{A}^*\|.
\label{d}\end{eqnarray}
In this article, we develop new refinements of the second inequality in (\ref{d}). We obtain inequalities involving numerical radius and spectral radius of the sum of the product operators, from which we achieve a nice refinement of the classical inequality $w(A)\leq \|A\|$. As application of the numerical radius inequalities, we give new bounds for the zeros of a complex monic polynomial which improve on the existing ones.

\section{Numerical radius inequalities}

We begin the section with the following lemmas.

\begin{lemma}\cite{kato}(Generalized Cauchy-Schwarz inequality)
 If $A\in\mathcal{B}(\mathscr{H})$ and $ 0\leq\alpha\leq1$, then $$ |\langle Ax,y\rangle|^2\leq\langle|A|^{2\alpha} x,x\rangle\langle|A^*|^{2(1-\alpha)}y,y\rangle$$ for all $x,y\in\mathscr{H}$. 
 
\label{lem1}\end{lemma}
\begin{lemma}\cite{a2}(Holder-McCarthy inequality) Let $A\in\mathcal{B}(\mathscr{H})$ be  positive. Then the following inequalities hold: $$\langle A^rx,x\rangle\geq \|x\|^{2(1-r)}\langle Ax,x\rangle^r,\,\,\,\,\, \textit{when $ r\geq1$}$$ $$\langle A^rx,x\rangle\leq\|x\|^{2(1-r)}\langle Ax,x\rangle^r,\,\,\,\,\, \textit{when $ 0\leq r\leq1$}$$  for any $x\in\mathscr{H}$.
\label{lem2}\end{lemma}

\begin{lemma}\cite{a3}(Buzano's inequality)
Let $ x,e,y\in\mathscr{H}$ with $\|e\|=1$, then $$|\langle x,e\rangle\langle e,y\rangle|\leq\frac{1}{2}\left(\|x\|\|y\|+|\langle x,y\rangle|\right).$$
\label{lem3}\end{lemma}

Now, we are in a position to present our results. First we develop the following upper bound for the numerical radius.

\begin{theorem}	\label{thre1}
	Let $ A\in\mathcal{B}(\mathscr{H}).$ Then \begin{eqnarray*}
		w^2(A)\leq\frac{1}{4} w^2 \left(|A|+i|A^*|\right)+\frac{1}{8}\left\||A|^2+|A^*|^2\right \|+\frac{1}{4}w\left(|A||A^*|\right).
	\end{eqnarray*}
\end{theorem}
\begin{proof}
	Let $x\in \mathscr{H}$ with $\|x\|=1$. Then we have
	 \begin{eqnarray*}
		&& |\langle Ax,x \rangle|^2\\
		&\leq&\langle |A|x,x\rangle \langle |A^*|x,x\rangle\,\,(\textit{by Lemma \ref{lem1}})\\&\leq&\frac{1}{4}\left(\langle |A|x,x\rangle+ \langle |A^*|x,x\rangle\right)^2\\&=&\frac{1}{4}\left(\langle |A|x,x\rangle^2+ \langle |A^*|x,x\rangle^2+2\langle |A|x,x\rangle \langle |A^*|x,x\rangle\right)\\&\leq&\frac{1}{4}\left\lbrace|\langle |A|x,x\rangle+i \langle |A^*|x,x\rangle|^2+\||A|x\|\||A^*|x\|+|\langle|A|x,|A^*|x\rangle|\right\rbrace\,(\textit{by Lemma \ref{lem3}})\\&\leq&\frac{1}{4}\left\lbrace|\langle (|A|+i|A^*|)x,x\rangle|^2+\frac{1}{2}\||A|x\|^2+\frac{1}{2}\||A^*|x\|^2+|\langle|A^*||A|x,x\rangle |\right\rbrace\\&\leq& \frac{1}{4} w^2\left(|A|+i|A^*|\right)+\frac{1}{8}\left\||A|^2+|A^*|^2\right\|+\frac{1}{4}w\left(|A||A^*|\right).
	\end{eqnarray*} 
Taking supremum over all $x\in \mathcal{H}$ with $\|x\|=1$, we get the desired inequality.
\end{proof}

Clearly, we see that 
\begin{eqnarray*}
	&&\frac{1}{4} w^2\left(|A|+i|A^*|\right)+\frac{1}{8}\left\||A|^2+|A^*|^2\right\|+\frac{1}{4}w\left(|A||A^*|\right)\\
		&\leq&\frac{1}{4}\left\||A|^2+|A^*|^2\right\|+\frac{1}{8}\left\||A|^2+|A^*|^2\right\|+\frac{1}{4}\left\||A||A^*|\right\|\\
		&=&\frac{3}{8}\left\||A|^2+|A^*|^2\right\|+\frac{1}{4}\left\|A^2\right\|\\
		&\leq&\frac{3}{8}\left\||A|^2+|A^*|^2\right\|+\frac{1}{8}\left\||A|^2+|A^*|^2\right\|\\
		&=&\frac{1}{2}\left\||A|^2+|A^*|^2\right\|.
	\end{eqnarray*} 
Thus, we would like to remark that the upper bound obtained in Theorem \ref{thre1} refines the second inequality in (\ref{d}). 
Next result reads as follows.
 
\begin{theorem}\label{th1}
Let $ X,Y\in\mathbb{B}(\mathscr{H})$, and  	 $ 0\leq\alpha\leq 1$, $ 0\leq\beta\leq 1$.  Then for each $x\in \mathscr{H}$ with $\|x\|=1$,  
\begin{eqnarray*}\label{eq2}
&& |\langle Xx,x\rangle \langle Yx,x\rangle| \\
&\leq & 
\frac{1}{4} \left \| \alpha |X|^2+(1-\alpha)|X^*|^2+\beta|Y|^2+(1-\beta)|Y^*|^2 \right \|+\frac{1}{8} \left \||X|^2+|Y^*|^2 \right \|+\frac{1}{4}w(YX).
\end{eqnarray*} 
\end{theorem}

\begin{proof} 
 	We have
 \begin{eqnarray*}
  	&&|\langle Xx,x\rangle \langle Yx,x\rangle|\\
  	&\leq&\frac{1}{4}\left\lbrace|\langle Xx,x\rangle|+|\langle Yx,x\rangle|\right\rbrace^2\\
  	&=&\frac{1}{4}\left\lbrace|\langle Xx,x\rangle|^2+|\langle Yx,x\rangle|^2+2|\langle Xx,x\rangle||\langle Yx,x\rangle|\right\rbrace\\
  	&\leq&\frac{1}{4}\left\lbrace\langle |X|^{2\alpha}x,x\rangle\langle |X^*|^{2(1-\alpha)}x,x\rangle+\langle |Y|^{2\beta}x,x\rangle\langle |Y^*|^{2(1-\beta)}x,x\rangle+2|\langle Xx,x\rangle||\langle x,Y^*x\rangle|\right\rbrace\\
  	&& \,\,\,\,\,\,\,\,\, (\textit{using Lemma \ref{lem1}})\\
  	&\leq&\frac{1}{4}\left\lbrace\langle |X|^{2}x,x\rangle^{\alpha}\langle |X^*|^{2}x,x\rangle^{(1-\alpha)}+\langle |Y|^{2}x,x\rangle^{\beta}\langle |Y^*|^{2}x,x\rangle^{(1-\beta)}+\| Xx\|\|Y^*x\|+|\langle Xx,Y^*x\rangle|\right\rbrace\\
  	&&\,\,\,\,\,\,\,\,\,\, (\textit{using Lemma \ref{lem2} and Lemma  \ref{lem3} })\\
  	&\leq&\frac{1}{4}\left\lbrace{\alpha}\langle |X|^{2}x,x\rangle+(1-\alpha)\langle |X^*|^{2}x,x\rangle+\beta\langle |Y|^{2}x,x\rangle+(1-\beta)\langle |Y^*|^{2}x,x\rangle\right\rbrace\\
  	&&+\frac{1}{4}\left\lbrace\frac{1}{2}\left(\langle |X|^2x,x\rangle+\langle |Y^*|^2x,x\rangle\right)+|\langle YXx,x\rangle|\right\rbrace\\
  	&\leq&\frac{1}{4}\| \alpha|X|^2+(1-\alpha)|X^*|^2+\beta|Y|^2+(1-\beta)|Y^*|^2\|+\frac{1}{8}\||X|^2+|Y^*|^2\|+\frac{1}{4}w(YX). 
  \end{eqnarray*}

\end{proof}

Applying the inequality in Theorem \ref{th1} we derive the following upper bound for the  numerical radius.

\begin{cor}\label{cor1}
If $ A\in\mathcal{B}(\mathscr{H})$, then
\begin{eqnarray*} 
w^2(A)\leq\frac{1}{4} \left \| \mu|A|^2+(2-\mu)|A^*|^2 \right\|+\frac{1}{8}\left\||A|^2+|A^*|^2\right\|+\frac{1}{4}w(A^2),
\end{eqnarray*}
for $0\leq \mu \leq 2.$
\end{cor}
\begin{proof}
 Putting $X=Y=A$ in Theorem \ref{th1}, and then taking supremum over all  $x\in \mathscr{H}$ with $\|x\|=1$, we get 
 \begin{eqnarray*} 
w^2(A)\leq\frac{1}{4} \left \| (\alpha+\beta)|A|^2+(2-\alpha-\beta)|A^*|^2 \right \|+\frac{1}{8} \left \||A|^2+|A^*|^2 \right \|+\frac{1}{4}w(A^2),
\end{eqnarray*}
for $0\leq \alpha,\beta \leq 1$. This implies the desired bound.
\end{proof}

It follows from Corollary \ref{cor1} that
\begin{eqnarray}\label{rem1} 
w^2(A)\leq \frac{1}{4} \, \underset{\mu\in[0,2]}{\min}  \left\| \mu|A|^2+(2-\mu)|A^*|^2 \right\|+\frac{1}{8}\left\||A|^2+|A^*|^2\right\|+\frac{1}{4}w(A^2).
\end{eqnarray}


\begin{remark}\label{rem2} 
Clearly, We have                                                                                                                                                                                                                                                                                                                                                                                                                                                                                                                                                                                                                                                                                                                                                                                                                                                                                                                                                                                                                                                                                                                                                                                                                                                                                                                    
 \begin{eqnarray*}
 & &\underset{\mu \in [0,2]}{\min} \,\, \frac{1}{4} \left \| \mu |A|^2+(2-\mu)|A^*|^2 \right\|+\frac{1}{8} \left \||A|^2+|A^*|^2 \right\|+\frac{1}{4}w(A^2)\\
&\leq &\frac{1}{4}\| |A|^2+|A^*|^2\|+\frac{1}{8}\||A|^2+|A^*|^2\|+\frac{1}{4}w(A^2)\,\,\, (\textit{by taking $\mu =1$})\\
&=&\frac{3}{8}\||A|^2+|A^*|^2\|+\frac{1}{4}w(A^2)\\
&\leq&\frac{3}{8}\||A|^2+|A^*|^2\|+\frac{1}{4}w^2(A)\\
&\leq&\frac{3}{8}\||A|^2+|A^*|^2\|+\frac{1}{8}\||A|^2+|A^*|^2\| \,\,\,\,(\textit{using the second inequality of (\ref{d})})\\
&=&\frac{1}{2}\||A|^2+|A^*|^2\|.
 	\end{eqnarray*}
Thus, we would like to remark that inequality \eqref{rem1} is stronger than that in (\ref{d}). We also note that the minimum value is not always attained for  $\mu =1.$
For example, consider the matrix $A=\begin{pmatrix}
	0 & 1 &0\\
	0 & 0 & 2\\
	0 & 0 & 0
\end{pmatrix}.$ Then, $\underset{\mu \in [0,2]}{\min} \, \left \| \mu |A|^2+(2-\mu)|A^*|^2 \right\|= \frac{32}{7}$ for $\mu =\frac87,$ and we see that
 \begin{eqnarray*}
  \frac{1}{4}	\underset{\mu \in [0,2]}{\min} \left \| \mu |A|^2+(2-\mu)|A^*|^2 \right\|+\frac{1}{8} \left \||A|^2+|A^*|^2 \right\|+\frac{1}{4}w(A^2)&=& \frac{113}{56}  \approx 2.01785714\\
 	&< & \frac52=\frac{1}{2}\||A|^2+|A^*|^2\|.
 \end{eqnarray*}

\end{remark}

To prove our next result we need the following two  lemmas. First one is a generalization of the inequality in Lemma \ref{lem1}, and the second one is known as Bohr's inequality.

\begin{lemma}(\cite[Th. 5]{a1}) Let $A,B\in\mathbb{B}(\mathscr{H})$ with $|A|B=B^*|A|$. Let $f,g$ be two non-negative continuous functions on $[0,\infty)$ such that $f(t)g(t)=t$ for all $t\geq0.$
 Then$$|\langle ABx,y\rangle|\leq r(B)\|f(|A|)x\|\|g(|A^*|)y\|,$$ for all $x,y\in\mathscr{H}.$\label{l1}\end{lemma}

 \begin{lemma}(\cite{n1})
 	 For $ i=1,2,\cdots,n$, let $a_i\geq0$.
 Then  $$ \left(\sum_{i=1}^{n}{a_i}\right)^p\leq n^{p-1}\sum_{i=1}^{n}{a_i^p},$$
 for all $p\geq1.$
\label{l2} \end{lemma}

By using the above lemmas we prove the following inequality involving numerical radius and spectral radius.

\begin{theorem}
Let  $A_i,B_i\in\mathbb{B}(\mathscr{H})$ be such that $|A_i|B_i=B_i^*|A_i|$ for  $ i=1,2,\cdots,n.$  Then $$ w^p\left(\sum_{i=1}^{n}{A_iB_i}\right)\leq\frac{n^{p-1}}{\sqrt{2}} w\left(\sum_{i=1}^{n}{r^p(B_i)\left( f^{2p}(|A_i|)+i g^{2p}(|A_i^*|)\right)}\right),$$ for all $p\geq1$.
\label{threm1}\end{theorem}
\begin{proof}
Let $x\in \mathscr{H}$ with $\|x\|=1$. Then we have 
\begin{eqnarray*}
&&\left|\langle\left(\sum_{i=1}^{n}{A_iB_i}\right) x,x\rangle\right|^p\\&=&\left|\sum_{i=1}^{n}\langle{A_iB_i} x,x\rangle\right|^p\\&\leq&\left(\sum_{i=1}^{n}\left|\langle{A_iB_i} x,x\right|\right)^p\\&\leq&\left(\sum_{i=1}^{n}{r(B_i)\|f(|A_i|)x\|\|g(|A_i^*|)x\|}\right)^p \,\,(\textit{by Lemma \ref{l1}})\\&=&\left(\sum_{i=1}^{n}{r(B_i)\langle f^2(|A_i|)x,x\rangle^\frac{1}{2}\langle g^2(|A_i^*|)x,x\rangle^\frac{1}{2}}\right)^p\\&\leq&\left(\sum_{i=1}^{n}{r(B_i)\frac{\langle f^2(|A_i|)x,x\rangle+\langle g^2(|A_i^*|)x,x\rangle}{2}}\right)^p\\&\leq&n^{p-1}\sum_{i=1}^{n}{r^p(B_i)\left(\frac{\langle f^2(|A_i|)x,x\rangle+\langle g^2(|A_i^*|)x,x\rangle}{2}\right)^p} \,\,\,(\textit{by Lemma \ref{l2}})\\&\leq&\frac{n^{p-1}}{2}\sum_{i=1}^{n}{r^p(B_i)\left(\langle f^2(|A_i|)x,x\rangle^p+\langle g^2(|A_i^*|)x,x\rangle^p\right)} \,\,\,(\textit{by convexity of $f(t)=t^p $})\\&\leq&\frac{n^{p-1}}{2}\sum_{i=1}^{n}{r^p(B_i)\left(\langle f^{2p}(|A_i|)x,x\rangle+\langle g^{2p}(|A_i^*|)x,x\rangle\right)} \,\,\,(\textit{by Lemma \ref{lem2}})\\&\leq&\frac{n^{p-1}}{\sqrt{2}} \left |\sum_{i=1}^{n}{r^p(B_i)\left(\langle f^{2p}(|A_i|)x,x\rangle+ i\langle g^{2p}(|A_i^*|)x,x\rangle\right)} \right |\\
&& \,\,\,\,\,\,\,\,\, \,\,\,\,\,\,\,\,\,\,\,\, \,\,\,\,\,\,\,\,\,\,\,\, \,\,\,\,\,\,\,\,\,\,\,\, \,\,\,\,\,\,\,\,\,\,\,\, \,\,\,\,\,\,\,\,\,\,\,\, \,\,\,\,\,\,\,\,\,\,\,\, \,\,\,\,\,\,\,\,\,\,\,\, \,\,\,(\textit{as $|a+b|\leq\sqrt{2}|a+ib|$ for all $a,b\in\mathbb{R}$})\\&\leq&\frac{n^{p-1}}{\sqrt{2}} \left | \sum_{i=1}^{n}{r^p(B_i)\langle\left( f^{2p}(|A_i|)+ i g^{2p}(|A_i^*|)\right) x,x\rangle} \right | \\&\leq&\frac{n^{p-1}}{\sqrt{2}} w\left(\sum_{i=1}^{n}{r^p(B_i)\left( f^{2p}(|A_i|)+i g^{2p}(|A_i^*|)\right)}\right).
\end{eqnarray*}
Now, taking supremum over all $x\in \mathscr{H}$, $\|x\|=1$ we get,
$$ w^p\left(\sum_{i=1}^{n}{A_iB_i}\right)\leq\frac{n^{p-1}}{\sqrt{2}} w\left(\sum_{i=1}^{n}{r^p(B_i)\left( f^{2p}(|A_i|)+i g^{2p}(|A_i^*|)\right)}\right).$$
 as desired.

\end{proof}

Observe that the inequality in Theorem \ref{threm1} indeed does not depend on the number $n $ of summands in the case $p=1$.
In particular, considering  $p=n=1$, $A_1=A$, $B_1=B$, $f(t)=g(t)=\sqrt{t}$ in Theorem \ref{threm1}, we get the following corollary.

\begin{cor}\label{corol1}
Let $A,B\in \mathbb{B}(\mathscr{H})$ be such that $|A|B=B^*|A|$. Then $$w(AB)\leq\frac{1}{\sqrt{2}}r(B)w(|A|+i|A^*|).$$

\end{cor}

In particular, for $B=I$ we have the following inequality (also obtained in \cite{BPBUL}): 
\begin{eqnarray}\label{pp25}
w(A)\leq\frac{1}{\sqrt{2}}w(|A|+i|A^*|).
\end{eqnarray}
	Note that the bound \eqref{pp25} refines that in (\ref{d}), see \cite[Remark 2.16]{BPBUL}.
Again, considering $B_i=I$ for  $ i=1,2,\cdots,n$ in  Theorem \ref{threm1} we have the following inequality for the sum of operators. 

\begin{cor}
	Let  $A_i\in\mathbb{B}(\mathscr{H})$ for  $ i=1,2,\cdots,n,$ and let $f,g$ be two non-negative continuous functions on $[0,\infty)$ such that $f(t)g(t)=t$ for all $t\geq0.$  Then $$ w^p\left(\sum_{i=1}^{n}{A_i}\right)\leq\frac{n^{p-1}}{\sqrt{2}} w\left(\sum_{i=1}^{n}{\left( f^{2p}(|A_i|)+i g^{2p}(|A_i^*|)\right)}\right),$$ for all $p\geq1$.
	\label{corpp}\end{cor}

In particular, for $n=1$ and $f(t)=g(t)=\sqrt{t}$ in Corollary \ref{corpp}, we get the following upper bound for the numerical radius.

\begin{cor}\label{corppp}
		If  $A\in\mathbb{B}(\mathscr{H})$, then 
		$$ w^p(A) \leq \frac{1}{\sqrt{2} } w(|A|^p+i|A^*|^p  ),$$ for all $p\geq 1.$
\end{cor}

It is easy to verify that  $\frac{1}{\sqrt{2} } w(|A|^p+i|A^*|^p  ) \leq \|A\|^p$ for all $p\geq 1.$ Therefore, we would like to remark that Corollary \ref{corppp} improves the classical bound $w(A) \leq \|A\|$ for all $p\geq 1.$

At the end of this section, we give a sufficient condition for the equality of $ w(A)=\frac{1}{2}\|A^*A+AA^*\|^{1/2}.$ For this purpose first we note the following known lemma.

\begin{lemma}\cite{a8}
	Let $A,B\in \mathbb{B}(\mathscr{H})$ be positive. Then, $\|A+B\|=\|A\|+\|B\|$ if and only if $\|AB\|=\|A\|\|B\|.$
	\label{lem7}\end{lemma}

\begin{theorem}\label{Thsuff}
	Let $A\in\mathbb{B}(\mathscr{H})$. Then $ \|A\|^4=\|\Re^2(A)\Im^2(A)\|$ implies  $$ w^2(A)=\frac{1}{4}\|A^*A+AA^*\|.$$
\end{theorem}
\begin{proof}
	We have 
	\begin{eqnarray*}
	\|A\|^4 &=&	\|\Re^2(A)\Im^2(A)\|\leq\|\Re^2(A)\|\|\Im^2(A)\|=\|\Re(A)\|^2\|\Im(A)\|^2\\&\leq&\frac{1}{2}\left(\|\Re(A)\|^4+\|\Im(A)\|^4\right) \leq \max\left(\|\Re(A)\|^4,\|\Im(A)\|^4\right)\\&\leq&w^4(A)\leq \|A\|^4.
	\end{eqnarray*}
	This implies that
	 \begin{eqnarray}
		\|\Re^2(A)\Im^2(A)\|=\|\Re(A)\|^2\|\Im(A)\|^2.
		\label{eq10}\end{eqnarray}
	Also, we have
	\begin{eqnarray}
		\frac{1}{2}\left(\|\Re(A)\|^4+\|\Im(A)\|^4\right)=\max\left(\|\Re(A)\|^4,\|\Im(A)\|^4\right)=w^4(A).
		\label{eq11}
	\end{eqnarray}
This implies that
	 \begin{eqnarray}
		\|\Re(A)\|=\|\Im(A)\|=w(A).
		\label{eq13}
	\end{eqnarray} 
Now, by using lemma \ref{lem7}, it follows from the identity (\ref{eq10}) that
 \begin{eqnarray*}
		\frac{1}{2}\|\Re^2(A)+\Im^2(A)\|&=&\frac{1}{2}\left(\|\Re^2(A)\|+\|\Im^2(A)\|\right)\\&=&\frac{1}{2}\left(\|\Re(A)\|^2+\|\Im(A)\|^2\right)\\
		&=& \|\Re(A)\|^2= w^2(A) \,\, (\textit{using \eqref{eq13}}).
	\end{eqnarray*}
This completes the proof.
	
\end{proof}

It should be mentioned here that the converse of Theorem \ref{Thsuff} is not true, in general. For example, we consider $A=\begin{pmatrix}
0 & 3 & 0 \\
0 & 0 & 0 \\
0 & 0 & 1\\
\end{pmatrix}.$ 
Then, $ w^2(A)=\frac{1}{4}\|A^*A+AA^*\|=\frac{9}{4},$ however $\|A\|^4\neq\|\Re^2(A)\Im^2(A)\|.$\\


\section{Estimation of zeros of polynomials}
Suppose  $ p(z) = z^n + a_nz^{n-1} + \ldots + a_2z + a_1 $ is a complex monic polynomial of degree $n\geq 2 $ and $a_1 \neq 0$.                                                                                                    
Location of the zeros of  $p(z)$
have been obtained by applying  numerical radius  inequalities  to Frobenius
companion matrix of the polynomial $p(z)$.  The Frobenius
companion matrix of the polynomial $p(z)$ is given by $$C_p=\begin{pmatrix}
-a_n & -a_{n-1} & .... & -a_2 & -a_1\\
1 & 0 & ...& 0 & 0\\
0 & 1 & ... & 0 & 0\\
\vdots & \vdots & \ddots & \vdots & \vdots\\
0 & 0 & .... & 1 & 0
\end{pmatrix}.$$
The characteristic polynomial of $ C_p$ is the polynomial $p(z)$. Thus, the zeros of $p(z)$ are exactly the eigenvalues of $C_p$, see \cite[p. 316]{a14}.
The square of $C_p$ is given by
$$ C_p^2=\begin{pmatrix}
b_n & b_{n-1} & ..... & b_3  & b_2 & b_1 \\
-a_n & -a_{n-1} & ....& -a_3 & -a_2 & -a_1\\
1 & 0 & ... & 0 & 0 & 0\\
0 & 1 & ... & 0 & 0 & 0\\
\vdots & \vdots & \ddots & \vdots & \vdots & \vdots\\
0 & 0 & .... & 1 & 0 & 0
\end{pmatrix},$$where $b_j = a_na_j-a_{j-1}$ for $j = 1, 2,\ldots, n,$ with $a_0 = 0$.\\ Also, $$ C_p^3=\begin{pmatrix}
c_n & c_{n-1} & .....& c_4 & c_3 & c_2 & c_1\\
b_n & b_{n-1} & ..... & b_4 & b_3  & b_2 & b_1 \\
-a_n & -a_{n-1} & .... &-a_4 & -a_3 & -a_2 & -a_1\\
1 & 0 & ...& 0 & 0 & 0 & 0\\
0 & 1 & ...& 0 & 0 & 0 & 0\\
\vdots & \vdots & \ddots & \vdots & \vdots & \vdots & \vdots\\
0 & 0 & ....&1 & 0 & 0 & 0
\end{pmatrix},$$ where $b_j = a_na_j -a_{j-1}$ and $c_j = -a_nb_j + a_{n-1}a_j - a_{j-2}$ for $j = 1, 2,\ldots,n ,$ with $a_0 = a_{-1} = 0$,\\
and $$ C_p^4=\begin{pmatrix}
d_n & d_{n-1} & ..... & d_5 & d_4 & d_3 & d_2 & d_1\\
c_n & c_{n-1} & ..... & c_5& c_4 & c_3 & c_2 & c_1\\
b_n & b_{n-1} & .....& b_5 & b_4 & b_3  & b_2 & b_1 \\
-a_n & -a_{n-1} & .... & -a_5 &-a_4 & -a_3 & -a_2 & -a_1\\
1 & 0 & ...& 0& 0 & 0 & 0 & 0\\
0 & 1 & ...& 0& 0 & 0 & 0 & 0\\
\vdots & \vdots & \ddots & \vdots& \vdots & \vdots & \vdots & \vdots\\
0 & 0 & ....&1 &0 & 0 & 0 & 0
\end{pmatrix},$$ where  $b_j = a_na_j - a_{j-1}$, $c_j = -a_nb_j + a_{n-1}a_j - a_{j-2}$, and $d_j = -a_nc_j - a_{n-1}b_{j-1} + a_{n-2}a_j - a_{j-3}$  for $j = 1, 2,\ldots, n,$  with
$a_0 = a_{-1} = a_{-2} = 0$.

The exact value of $\|C_p\|$ is well known (see in \cite{a7}), it is given by
\begin{eqnarray}
	\|C_p\|=\sqrt\frac{{\alpha+1+\sqrt{(\alpha+1)^2-4|a_1|^2}}}{2},
\end{eqnarray} where $\alpha=\sum_{j=1}^{n}{|a_j|^2}$.

An estimation of $\|C_p^2\|$ obtained in \cite{a21} is as follows
 \begin{eqnarray}\label{ppp26}
	\|C_p^2\|\leq\sqrt{\frac{\delta+1+\sqrt{(\delta-1)^2+4\delta'}}{2}},
	\label{eq5}\end{eqnarray}
where $\delta=\frac{1}{2}\left(\alpha+\beta+\sqrt{(\alpha-\beta)^2+4|\gamma|^2}\right)$ and  $\delta'=\frac{1}{2}\left(\alpha'+\beta'+\sqrt{(\alpha'-\beta')^2+4|\gamma'|^2}\right)$, $\alpha=\sum_{j=1}^{n}{|a_j|^2}$, $\beta=\sum_{j=1}^{n}{|b_j|^2}$, $\alpha'=\sum_{j=3}^{n}{|a_j|^2}$, $\beta'=\sum_{j=3}^{n}{|b_j|^2}$, $\gamma=-\sum_{j=1}^{n}{\bar{a_j}b_j}$, $\gamma'=-\sum_{j=3}^{n}{\bar{a_j}b_j}.$

We note that  
\begin{eqnarray*}
	\|C_p^2\|^{\frac12}\leq \left({\sqrt{\frac{\delta+1+\sqrt{(\delta-1)^2+4\delta'}}{2}}}\right)^{1/2} \leq \sqrt\frac{{\alpha+1+\sqrt{(\alpha+1)^2-4|a_1|^2}}}{2}=\|C_p\|.
\end{eqnarray*}

Motivated by the above estimation, here we will obtain an estimation of $\|C_p^4\|^{1/4}$. For this purpose first we note the following norm inequality for the sum of two positive operators.

\begin{lemma}\cite{a8}
If $ A,B\in\mathbb{B}(\mathscr{H})$ are positive, then $$ \|A+B\|\leq\frac{1}{2}\left(\|A\|+\|B\|+\sqrt{ \left(\|A\|-\|B\|\right)^2+
	4\left \|A^\frac{1}{2}B^\frac{1}{2} \right\|^2}\right).$$
\label{lem4}\end{lemma}

Now, we are in a position to obtain an estimation of $\|C_p^4\|^{1/4}.$ Let
 $ C_p^4= R+S+T, $ where $$ R=\begin{pmatrix}
d_n & d_{n-1} & ..... & d_5 & d_4 & d_3 & d_2 & d_1\\
c_n & c_{n-1} & ..... & c_5& c_4 & c_3 & c_2 & c_1\\
0 & 0 & .....& 0 & 0 & 0  & 0 & 0\\
\vdots & \vdots & \ddots & \vdots& \vdots & \vdots & \vdots & \vdots\\
0 & 0 & ....&0 &0 & 0 & 0 & 0
\end{pmatrix},$$ \\ $$ S=\begin{pmatrix}
0 & 0 & ..... & 0 & 0 & 0 & 0 & 0\\
0 & 0 & ..... & 0 & 0 & 0 & 0 & 0\\
b_n & b_{n-1} & .....& b_5 & b_4 & b_3  & b_2 & b_1 \\
-a_n & -a_{n-1} & .... & -a_5 &-a_4 & -a_3 & -a_2 & -a_1\\
0 & 0 & ...& 0& 0 & 0 & 0 & 0\\
\vdots & \vdots & \ddots & \vdots& \vdots & \vdots & \vdots & \vdots\\
0 & 0 & ....&0 &0 & 0 & 0 & 0
\end{pmatrix}$$ and $$ T=\begin{pmatrix}
0 & 0 & ..... & 0 & 0 & 0 & 0 & 0\\
0 & 0 & ..... & 0 & 0 & 0 & 0 & 0\\
0 & 0 & .....& 0 & 0 & 0  & 0 & 0 \\
0 & 0 & .... & 0 & 0 & 0 & 0 & 0\\
1 & 0 & ...& 0& 0 & 0 & 0 & 0\\
0 & 1 & ...& 0& 0 & 0 & 0 & 0\\
\vdots & \vdots & \ddots & \vdots& \vdots & \vdots & \vdots & \vdots\\
0 & 0 & ....&1 &0 & 0 & 0 & 0
\end{pmatrix}.$$
 Now,
\begin{eqnarray*}
\|C_p^4\|^2&=&\| R+S+T\|^2\\&=&\|(R+S+T)^*(R+S+T)\|\\&=&\|R^*R+S^*S+T^*T\|\,(\textit{since      $ R^*S=R^*T=S^*R=S^*T=T^*R=T^*S=0$})\\&\leq&\|R^*R+S^*S\|+\|T^*T\|\\&\leq&\frac{1}{2}\left(\|R\|^2+\|S\|^2+\sqrt{\left(\|R\|^2-\|S\|^2\right)^2+4\|RS^*\|^2}\right)+1 \,\,(\textit{using Lemma \ref{lem4}}).
\end{eqnarray*}
By simple calculations, we have
 \begin{eqnarray*}
\|R\|^2&=&\|R^*R\|=\|RR^*\|\\&=&\frac{1}{2}\left(\alpha_1+\beta_1+\sqrt{(\alpha_1-\beta_1)^2+4|\gamma_1|^2}\right)=\delta_1,
\end{eqnarray*} where $\alpha_1=\sum_{j=1}^{n}|d_j|^2$, $\beta_1=\sum_{j=1}^{n}|c_j|^2$ , $\gamma_1=\sum_{j=1}^{n}d_j\bar{c_j},$  
\begin{eqnarray*}
\|S\|^2&=&\|S^*S\|=\|SS^*\|\\&=&\frac{1}{2}\left(\alpha+\beta+\sqrt{(\alpha-\beta)^2+4|\gamma|^2}\right)=\delta,
\end{eqnarray*} where $\alpha=\sum_{j=1}^{n}|a_j|^2$, $\beta=\sum_{j=1}^{n}|b_j|^2$ , $\gamma=-\sum_{j=1}^{n}b_j\bar{a_j},$
\begin{eqnarray*}
&&\|RS^*\|^2\\&=&\frac{1}{2}\left(|\gamma_2|^2+|\gamma_3|^2+|\gamma_4|^2+|\gamma_5|^2+\sqrt{\left((|\gamma_2|^2+|\gamma_3|^2)-(|\gamma_4|^2+|\gamma_5|^2)\right)^2+4|\gamma_2\bar{\gamma_4}+\gamma_3\bar{\gamma_5}|^2}\right)\\
&=&\delta_2,
\end{eqnarray*}
where $\gamma_2=\sum_{j=1}^{n}d_j\bar{b_j}$, $\gamma_3=\sum_{j=1}^{n}d_j\bar{a_j}$, $\gamma_4=\sum_{j=1}^{n}c_j\bar{b_j}$, $\gamma_5=\sum_{j=1}^{n}c_j\bar{a_j}$.\\
Therefore, 
 \begin{eqnarray}
\|C_p^4\|\leq  \sqrt{ \frac{1}{2}\left(\delta_1+\delta+\sqrt{(\delta_1-\delta)^2+4\delta_2}\right)+1}.
\label{eq4}\end{eqnarray} 

We observe that the estimation of $\|C_p^4\|^{1/4}$ in \eqref{eq4} is incomparable with the existing estimation of $\|C_p^2\|^{1/2}$ in \eqref{ppp26}.
In the following theorem we derive an upper bound for the spectral radius of the Frobenius companion matrix $C_p$, by using the estimations in (\ref{eq5}) and (\ref{eq4}).
\begin{theorem}
	The following inequality holds:
	$$r(C_p)\leq\left\lbrace\frac{1}{4}\left(\frac{\delta+1+\sqrt{(\delta-1)^2+4\delta'}}{2}\right)+\frac{3}{4}\left(\frac{1}{2}\left(\delta_1+\delta+\sqrt{(\delta_1-\delta)^2+4\delta_2}\right)+1\right)^\frac{1}{2}\right\rbrace^\frac{1}{4},$$ where $\delta'=\frac{1}{2}\left(\alpha'+\beta'+\sqrt{(\alpha'-\beta')^2+4|\gamma'|^2}\right)$, \\$\delta=\frac{1}{2}\left(\alpha+\beta+\sqrt{(\alpha-\beta)^2+4|\gamma|^2}\right)$,\\$\delta_1=\frac{1}{2}\left(\alpha_1+\beta_1+\sqrt{(\alpha_1-\beta_1)^2+4|\gamma_1|^2}\right)$,\\$\delta_2= \frac{1}{2}\left(|\gamma_2|^2+|\gamma_3|^2+|\gamma_4|^2+|\gamma_5|^2+\sqrt{\left((|\gamma_2|^2+|\gamma_3|^2)-(|\gamma_4|^2+|\gamma_5|^2)\right)^2+4|\gamma_2\bar{\gamma_4}+\gamma_3\bar{\gamma_5}|^2}\right)$,\\ $\alpha'=\sum_{j=3}^{n}{|a_j|^2}$, $\beta'=\sum_{j=3}^{n}{|b_j|^2}$, $\gamma'=-\sum_{j=3}^{n}{\bar{a_j}b_j},$\\$\alpha=\sum_{j=1}^{n}|a_j|^2$, $\beta=\sum_{j=1}^{n}|b_j|^2$ , $\gamma=-\sum_{j=1}^{n}b_j\bar{a_j}$,\\  $\alpha_1=\sum_{j=1}^{n}|d_j|^2$, $\beta_1=\sum_{j=1}^{n}|c_j|^2$ , $\gamma_1=\sum_{j=1}^{n}d_j\bar{c_j}$,\\$\gamma_2=\sum_{j=1}^{n}d_j\bar{b_j}$, $\gamma_3=\sum_{j=1}^{n}d_j\bar{a_j}$, $\gamma_4=\sum_{j=1}^{n}c_j\bar{b_j}$, $\gamma_5=\sum_{j=1}^{n}c_j\bar{a_j}.$ 
\label{th2}\end{theorem}

\begin{proof}
	Let $A\in \mathbb{B}(\mathscr{H})$. 
Putting $A=A^2$ in the inequality  $w^2(A)\leq \frac14 \|A^*A+AA^*\|+\frac12 w(A^2)$ (see \cite[Th. 2.4]{a17}),
 we get
\begin{eqnarray*}
	w^2(A^2)&\leq&\frac{1}{4} \left\||A^2|^2+|{(A^*)}^2|^2\right\|+\frac{1}{2}w(A^4).
\end{eqnarray*}
It follows that 
\begin{eqnarray*} 
	r^2(A)=r(A^2)\leq w(A^2)\leq\left\lbrace\frac{1}{4} \left \||A^2|^2+|{(A^*)}^2|^2\right\|+\frac{1}{2}w(A^4)\right\rbrace^{\frac{1}{2}},
\end{eqnarray*}
i.e.,  
\begin{eqnarray} \label{spec1}
	r(A)\leq\left\lbrace\frac{1}{4} \left\||A^2|^2+|{(A^*)}^2|^2\right\|+\frac{1}{2}w(A^4)\right\rbrace^{\frac{1}{4}}.
\end{eqnarray}	
Now, it follows from \eqref{spec1} and the inequality $\|C_p^*C_p+C_pC_p^*\| \leq \|C_p\|^2+\|C_p^2\|$ (see \cite[Remark 3.9]{a19}) that
\begin{eqnarray*}
r(C_p) &\leq&\left\lbrace\frac{1}{4} \left\||C_p^2|^2+|{(C_p^*)}^2|^2\right\|+\frac{1}{2}w(C_p^4)\right\rbrace^{\frac{1}{4}}\\
&\leq&\left\lbrace\frac{1}{4}(\|C_p^2\|^2+\|C_p^4\|)+\frac{1}{2}\|C_p^4\|\right\rbrace^{\frac{1}{4}}\\
&\leq& \left\lbrace\frac{1}{4} \left\|C_p^2\right\|^2+\frac{3}{4}\left\|C_p^4\right\|\right\rbrace^\frac{1}{4}. 
\end{eqnarray*}
Therefore, the required inequality follows by using the estimations in (\ref{eq5}) and (\ref{eq4}). 
\end{proof}

By using the fact
$|\lambda_j(C_p)|\leq r(C_p)$, where $\lambda_j(C_p)$ is the $j$-th eigenvalue of $C_p$, we infer the following estimation for the zeros of the polynomial $p(z)$.
\begin{theorem}
If $z$ is any zero of $p(z)$, then $$|z|\leq\left\lbrace\frac{1}{4}\left(\frac{\delta+1+\sqrt{(\delta-1)^2+4\delta'}}{2}\right)+\frac{3}{4}\left(\frac{1}{2}\left(\delta_1+\delta+\sqrt{(\delta_1-\delta)^2+4\delta_2}\right)+1\right)^\frac{1}{2}\right\rbrace^\frac{1}{4},$$
where $\delta$, $\delta_1$, $\delta_2$ and $\delta'$ are same as in Theorem \ref{th2}. 
\label{th3}\end{theorem}

Applying the spectral mapping theorem, we conclude that if $z$ is any zero of $p(z)$ then $|z|\leq\|C_p^4\|^\frac{1}{4}$. 
Thus, by using the inequality (\ref{eq4}) we achieve  another new estimation for the zeros of $p(z)$.
\begin{theorem}
If $z$ is any zero of $p(z)$, then $$ |z|\leq\left\lbrace\frac{1}{2}\left(\delta_1+\delta+\sqrt{(\delta_1-\delta)^2+4\delta_2}\right)+1\right\rbrace^\frac{1}{8},$$ where $\delta$, $\delta_1$ and $\delta_2$ are given in Theorem \ref{th2}.
\label{th4}\end{theorem}

Again, putting $A=A^2$ in the inequality $w(A)\leq \frac{1}{2} \left( \|A\|+ \|A^2\|^{\frac12} \right )$ (see  \cite[Th. 1]{a21}),
and proceeding as \eqref{spec1}, we get
\begin{eqnarray}\label{spec2}
	r(A)&\leq&\left\lbrace\frac{1}{2
	}\|A^2\|+\frac{1}{2}\|A^4\|^\frac{1}{2}\right\rbrace^{\frac{1}{2}}.
\end{eqnarray}
Proceeding similarly  as in Theorem \ref{th2} we obtain the following estimation by using the inequalities in (\ref{spec2}), (\ref{eq5}) and (\ref{eq4}).

\begin{theorem}
If $z$ is any zero of $p(z)$, then $$ |z|\leq\left\lbrace\frac{1}{2}\sqrt{\frac{\delta+1+\sqrt{(\delta-1)^2+4\delta'}}{2}}+\frac{1}{2}\left(\frac{1}{2}\left(\delta_1+\delta+\sqrt{(\delta_1-\delta)^2+4\delta_2}\right)+1\right)^\frac{1}{4}\right\rbrace^\frac{1}{2},$$
\label{th5}
\end{theorem}
where $\delta$, $\delta_1$, $\delta_1$ and $\delta'$ are given in Theorem \ref{th2}.

Finally, we compare the bounds obtained here for the zeros of $p(z)$ with the existing ones. First we note some well known existing bounds.
Let $z$ be any zero of $ p(z)$. Then

Linden \cite{a12} obtained that
$$|z|\leq\frac{|a_n|}{n}+\left(\frac{n-1}{n}\left(n-1+\sum_{j=1}^{n}{|a_j|^2}-\frac{|a_n|^2}{n}\right)\right)^\frac{1}{2}.$$

Montel \cite[Th. 3]{a11} obtained that
$$|z|\leq\max\left\lbrace1,|a_1|+\cdots+|a_n|\right\rbrace.$$

Cauchy \cite{a14} obtained that 
$$|z|\leq 1+\max\left\lbrace|a_1|,\cdots,|a_n|\right\rbrace.$$

Kittaneh \cite{a15} proved that
$$|z|\leq\frac{1}{2}\left(|a_n|+1+\sqrt{(|a_n|-1)^2+4\sqrt{\sum_{j=1}^{n-1}{|a_j|^2}}}\right).$$

Fujii and Kubo \cite{a13} obtained that
$$ |z|\leq \cos\frac{\pi}{n+1}+\frac{1}{2}\left(|a_n|+\sqrt{\sum_{j=1}^{n}{|a_j|^2}}\right).$$

Bhunia and Paul \cite[Th. 2.6]{a16} proved that 
$$|z|^2\leq\cos^2\frac{\pi}{n+1}+|a_{n-1}|+\frac{1}{4}\left(|a_n|+\sqrt{\alpha}\right)^2+\frac{1}{2}\sqrt{\alpha-|a_{n}|^2}+\frac{1}{2}\sqrt{\alpha},$$
where $ \alpha=\sum_{j=1}^{n}{|a_j|^2}.$

We consider a polynomial $ p(z)=z^3+z^2+\frac{1}{2}z+1$. Different upper bounds for the modulus of the zeros of this polynomial, mentioned above, are as shown in the following table.
\begin{center}
\begin{tabular}{|c|c|}
\hline
Linden \cite{a12} & 1.9492 \\
Montel\cite{a11} & 2.5 \\
Cauchy\cite{a14} & 2 \\
Kittaneh\cite{a15} & 2.0547\\
Fujii and Kubo\cite{a13} & 1.9571\\
Bhunia and Paul\cite{a16} & 1.96761\\
\hline
\end{tabular}
\end{center}
However, Theorem \ref{th3} gives $|z|\leq 1.38047091798$, Theorem \ref{th4} gives $|z|\leq 1.3798438819$ and Theorem \ref{th5} gives $|z|\leq 1.381095966$, which are better than the above mentioned bounds.

\noindent \underline{\textbf{Statements \& Declarations}}: \\
\textbf{Funding}. The authors declare that no funds, grants, or other support were received during the preparation of this manuscript.\\
\textbf{ Competing interests.} The authors have no relevant financial or non-financial interests to disclose. \\
\textbf{Data availability statements.} Data sharing not applicable to this article as no datasets were generated or analysed during the current study.\\
\textbf{Author Contributions.} All authors have contributed equally in the preparation of the manuscript. 
\bibliographystyle{amsplain}

\end{document}